\documentclass[12pt,reqno]{amsart}
\usepackage{amsmath,amssymb,amsthm,hyperref}

\numberwithin{equation}{section}

\newtheorem{theorem}{Theorem}[section]
\newtheorem{proposition}[theorem]{Proposition}
\newtheorem{lemma}[theorem]{Lemma}

\theoremstyle{definition}

\theoremstyle{remark}
\newtheorem{remark}[theorem]{Remark}

\def\bZ{{\mathbb {Z}}}

\def\bC{{\mathbb {C}}}
\def\bP{{\mathbb {P}}}

\def\bD{{\mathbb {D}}}

\def\bI{{\mathbb {I}}}

\def\bD{{\mathbb {D}}}

\def\pO{{\mathcal O}}

\begin{document}

\title[Theta-functions for singular curves]{Theta-functions for singular curves}

\author[I. Biswas]{Indranil Biswas}

\address{Department of Mathematics, Shiv Nadar University, NH91, Tehsil
Dadri, Greater Noida, Uttar Pradesh 201314, India}

\email{indranil.biswas@snu.edu.in, indranil29@gmail.com}

\author[J. Hurtubise]{Jacques Hurtubise}

\address{Department of Mathematics, McGill University, Burnside
Hall, 805 Sherbrooke St. W., Montreal, Que. H3A 2K6, Canada}

\email{jacques.hurtubise@mcgill.ca}


\keywords{Singular Riemann surface, theta-functions; there are no data associated to this paper}

\date{}

\begin{abstract}
Let $X$ be an irreducible singular Riemann surface, with desingularisation $\widetilde X$. The generalised Jacobian
$J(X)$ of $X$ fibers over the Jacobian $J(\widetilde{X})$ of $\widetilde X$, and there is an Abel map
$A$ of $\widetilde X$ to $J(X)$, lifting the Abel map to $J(\widetilde X)$. We build a theta function on
a compactification of the generalised Jacobian $J(X)$ (giving a section of a suitable positive line
bundle). The translation action on $J(X)$ then yields all line bundles of that degree, and the translates
of the theta function, restricted to $A(\widetilde X)$, give a ``universal section" of the line bundles
of that degree over $X$. This extends to the singular case a classical result of Riemann.
\end{abstract}

\maketitle

\section{Introduction}\label{sec0}

Theta functions and their various generalizations have a wide variety of uses. In their most classical 
form, on the Jacobian of a compact Riemann surface, they give --- in some sense --- a universal construction
of the function theory of the curve. The curve embeds into its Jacobian by the well-known
Abel map: $A\, :\, X\,\hookrightarrow\, J(X)$. The theta function represents a section of a line bundle $L$ on the curve, in a 
degree in which the translation action on $J(X)$ gives all line bundles of that degree. Thus translating $A(X)$ by $T_a$, and pulling back the 
theta-function, we get a section of the pull back $A^* T_a^* L$ on the curve, represented as a quasi-periodic function. From this, taking products, 
quotients, etc of these various translates, we can recreate all the function theory of the curve.
For beautiful expositions of all this, 
the reader is referred to several books of Mumford, notably the three books of Tata lectures \cite{Mu1, 
Mu2, Mu3}, and his monograph ``Curves and their Jacobians'' \cite{Mu4}.

One particular use of these functions lies in the solution of various problems in integrable systems, most 
notably when one considers the flows in several parameters associated to equations of KP type. There are many 
ways of doing this (see, among many others, \cite{B, K, Du}); one simple version, expounded in \cite{AHH}, is 
to consider Lax equations
\begin{equation}\label{spectral}
\frac{ dL(\lambda) }{dt}\ =\ [P(L(\lambda)),\, L(\lambda)]\end{equation}
where $L(\lambda)\, =\, L(\lambda,\, t)$ is some flow (in $t$) of matricial rational functions in $\lambda$. Here $P$ is
some function of $L(\lambda)$ --- typically a projection. The flow has as an invariant the spectral curve $S$ in
the $(\lambda,\, z)$-plane given by $\det(L\lambda - z\bI) \,=\, 0$, and the cokernel of $(L(\lambda) - z\bI)$ will typically be a line bundle $E$ supported on $S$. The cases of interest   give   linear flows of line bundles on the Jacobian, and one can then write out the solutions with a bit of
effort in terms of the $\theta$--functions --- first writing out a basis of sections of $E$ and from there an expression for $L$. The scheme has various generalisations, for example the  generalised Hitchin systems \cite{BR, Bo, Ma}, where the parameter $\lambda$ lies on a Riemann surface, and $L(\lambda)$ is a section of some twist of the endomorphisms of a bundle.

This method has been very successful in producing quasi-periodic solutions to KP-type equations (combining several flows of the same type). The 
solutions which piqued interest in these equations many years ago were, however, solitons.  They 
correspond here to a degeneration of the curve into a rational, typically nodal, curve, with the period 
lattice of the Jacobian expanding out to infinity in some directions. Indeed, various ingenious 
degenerations from the quasi-periodic to the solitonic flows have been presented over the years.
A small, recent, sample could include \cite{Ab},
\cite{Ga}, \cite{Ko}; there remains the fact that the basic geometry of \eqref{spectral} lies in the limit, and it would 
be nice to solve in these terms.

The question then arises of understanding (\ref{spectral}) with a singular spectral curve. There already 
exists a suitable generalised Jacobian (see \cite{AK}, \cite{Ro}, \cite{Ca}, \cite{Gi}, \cite{OS})
and its concomitant Abel map, defined on the desingularisation $\widetilde X$ of $X$, the question then 
arises of how to generalise the theta function. Indeed, there are already several examples of this in the 
literature, starting as far as we can tell with a paper of Clebsch \cite{Cl}, dealing with a curve with a 
simple node, and other cases have followed.

We propose here to deal with the general case; indeed, once one sees what is required, finding a generalisation 
is not too difficult. If we look at the usual theta function, we have that translating the image of the curve in its Jacobian by the Abel map,
we obtain by pull back
\begin{enumerate}
\item a realisation of sections of  regular line bundles in degree equal to the genus as quasi-periodic
functions, with suitable automorphy factors, and

\item an actual section of that line bundle, given by the translate (modulo the Riemann constant)  of the theta function, with a divisor given by the
vanishing of the theta-function.
\end{enumerate}
Our aim is to obtain a similar construction here, and indeed our theta-functions will exhibit the same quasi-periodicity as in the smooth case.

The Jacobian of the singular curves has  a higher dimension  compared to that of its desingularisation. Roughly, pulling back from the singular curve to its desingularisation,  the various rank one locally free sheaves (line bundles) on the singular curve correspond on the desingularisation to different equivalence classes of trivialisations at the preimage of the singular points of line bundles. In our case these trivialisations will  be simply the pull-backs of the shifts of our generalised theta-functions.

We will in the course of our construction be compactifying the generalised Jacobian; 
there are several ways of doing this in a natural fashion, adding on torsion free sheaves, and so on.  Our 
compactification will be much more elementary, and we are not sure if it has much geometric meaning. The 
aim is to have a compactification which works, i.e., with a theta-function that gives us what we want, 
that is some sort of universal section of line bundles over the curve.  The theta-function will be a section of a line bundle which is non-trivial on the fibers of this compactification, and so when pulled back to the desingularised curve will be a section of degree greater than that of the line bundle for the
usual theta function. This is quite normal, as we are looking for sections on the singular curve, and this has higher arithmetic genus.

The construction we have is fairly explicit, at least in the sense that ordinary theta functions are 
explicit: there will be in the construction, periods of Abelian integrals, generalisation of the Riemann 
constant, and indeed the basic ingredient will be the standard theta function associated to the 
desingularisation. We will proceed in steps, starting with simple cases, to illustrate what is involved.

\section{The generalised Jacobian}

Let $X$ be an irreducible singular algebraic curve, with singular points $x_1,\,\cdots,\,x_s$, and let
$\pi\ :\ \widetilde X\ \longrightarrow\ X$ be the desingularisation of $X$; the
desingularisation $\widetilde X$ is connected. The genus of $\widetilde X$ is denoted by $g$. We will be
interested in the 
generalized Jacobian $J(X)$ of $X$ parametrizing all locally free rank one sheaves of degree zero on $X$.

There will be an exact sequence 
\begin{equation}\label{e1}
0\,\longrightarrow\, Q^S(X) \,\longrightarrow\, J(X)\, \longrightarrow\, J(\widetilde X)
\,\longrightarrow\, 0
\end{equation}
of Abelian groups, with $Q^S(X)$ representing the different ways of ``collapsing" a line bundle over
the preimage of the singular points in $\widetilde X$ to a locally free module on $X$.

More precisely, let the  preimages  of the singularities $x_i$,\, $i\,=\,1,\,\cdots,\,s$, in $X$ be
denoted by $p_{i,j}\,\in\, \widetilde X$,\, $i\,=\, 1,\,\cdots ,\,s$,\, $j \,=\, 1,\,\cdots ,\,j_i$.
We note that $Q^S(X)$ is the quotient group
\begin{equation}\label{e2}
Q^S(X)\ =\ \left(\bigoplus_{i,j} \pO^*_{p_{i,j}}\right){\Big/}\left(\bigoplus_i \pi^*\pO^*_{x_i}\right).
\end{equation} 
In essence, elements of this $Q^S(X)$ represent the different local trivializations of a line
bundle around the points $p_{i,j}$ which get identified on $X$, modulo the action of changes of
trivialisation on $X$. As a space, $Q^S(X)$ is a product of factors $\bC$ and $\bC^*$, with
$\sum_{i=1}^s (j_i-1)$ copies of $\bC^*$ representing the quotients of the different values of
functions at the $p_{i,j}$ corresponding to a given $x_i$, and the factors $\bC$ representing
the higher order terms. As a group, $Q^S(X)$ will be an iterated extension of these
copies of $\bC$ and $\bC^*$. Note that $Q^S(X)$ can be given instead by quotients of ``finite order''
rings $\bigoplus_{i,j}( \pO^*_{p_{i,j}})_{<n}$. This can be
seen from the fact that expanding order by order in coordinates centred at the $p_{i,j}$ the quotient
eventually stabilises. 

If our singular curve $X$ is desingularized by $\pi\,:\, \bP^1\,\longrightarrow\, X$, the
generalised Jacobian is then just $Q_S(X)$. As an example, for the nodal cubic in $\bP^2$, with two points
$z\,=\,0,\,1$ in the desingularisation of the node, the quotient $Q^S$ will be the quotient by the
diagonal $(\bC^*\oplus \bC^*)/ \bC^*= \{(\lambda,\, \lambda)\, \big\vert\,\, \lambda\, \in\,\bC^*\}$
of the values at $0,\,1$. For the cubic $y^2 \,=\, x^3$ with a cusp, the desingularisation gives a single
point $p$ (say $z\,=\,0$) as the preimage of the singular point, and the quotient is
$$\{a_0 + a_1z +a_2z^2+\ldots\, \big\vert\, a_0\,\neq\, 0\}/\{b_0 + b_2z^2+b^3z_3+\ldots |b_0\,\neq
\, 0\}\ =\ \bC .$$
 
Infinitesimally, the Lie algebra 
 \begin{equation}\label{e3}
q^S(X)\ =\ \left(\bigoplus_{i,j} \pO_{p_{i,j}}\right){\Big/}\left(\bigoplus_i \pi^*\pO_{x_i}\right)
\end{equation}
of the group $Q^S(X)$ is isomorphic to the Abelian Lie algebra $\bC^n$. If $\{F_\alpha\}_{\alpha\in A}$
is a set of germs of non-vanishing functions generating $Q^S(X)$, then $q^S(X) $ in \eqref{e3}
is generated (additively) by
$$f_a \ =\ \log(F_\alpha)\ \in\ \left(\bigoplus_{i,j}\pO_{p_{i,j}}\right).$$
We note that as $\bigoplus_{i,j} \pO_{p_{i,j}}$ is generated by monomials, each
supported at only one of the $p_{i,j}$, the same holds for $q^S(X)$. Explicitly if $p_{i,j}\,\in\,
\widetilde X$,\,   $j \,=\, 1,\,\cdots ,\,j_i$ are our points mapped to $x_i$, we can set $f_{i,j}$,\,
$j\,=\, 2,\,\cdots ,\,j_i$, to be $1$ at $p_{i,j}$, and $0$ at the other $p_{i,k}$. For the generators
of higher order, we can choose $f_{i,j,h} \,=\, z^{n_{i,j, h}}$,\, $h\,=\, 1,\,\cdots,\, h_{i,j}$, supported
at $p_{i,j}$, and zero elsewhere; here $z$ is a holomorphic coordinate centred at the point $p_{ij}$.

We can represent  the  dual $q^S(X)^*$  by one-forms $\omega^S_{i,j}$,\, $\omega^S_{i,j,h}$ with poles at
$p_{i,j}$, with a pairing
$$\langle \omega^S \, ,\ f\rangle\ =\ \sum_{i,j} Res_{p_{i,j}}(f\omega^S).$$
Thus, dually to $f_{i,j}$, we have $\omega_{i,j} $ with only  a pole of order 1 at $p_{i,j}$,\, $p_{i,1}$ with residues $1, -1$ respectively, and  dually to $f_{i,j,h}$, a $1$-form $\omega_{i,j, h} $ with no residues, and a single pole at $p_{i,j} $ with polar part $z^{-1-n_{i,j, h} }dz$.
 
When  $\widetilde X \,=\, \bP^1$, these one-forms can be quite explicit: If $z$ is now a global holomorphic
coordinate function, there will be one-forms with simple poles 
$$\omega^S_{i, j}\ =\ \left(\frac{-1}{z-p_{i, 1}}+ \frac{1}{z-p_{i, j}}\right)dz ,\ \ j\,=\, 2,\,\cdots ,\, j_i$$
which detect whether the functions have the same values at $p_{i,j}$ and $p_{i,1}$ as well as the one-forms
$$\omega^{S}_{i,j,h} =\   (z-p_{i,j})^{-1-n_{i,j,h}}dz,$$
dual to the higher order generators. Note that normally there would be constraints on these one-forms, in that they have to pair to zero 
with functions lifted from neighbourhoods of the singular points in $X$. In fact, as we will be
working with fixed generators, we will not be concerned by this.

For our two examples, we have the one-forms $(\frac{1}{z} -\frac{1}{z-1}) dz$ for the nodal cubic, and
$\frac{dz}{z^2} $ for the one with a cusp. In the more general case of an arbitrary genus
for a desingularisation $\widetilde X\,\longrightarrow\, X$ of a singular curve $X$,
the only constraint on specifying the polar parts of one-forms on a Riemann surface is that the residues
add up to zero, and there will be global one-forms with poles at the points $p_{i,j}$ spanning $q^S(X)^*$. 

For a general $\widetilde X$, choosing a standard homology basis of A-cycles and B-cycles for $H_1(\widetilde X,\,
\bZ)$, we can normalise the one-forms spanning the cotangent space of $J(X)$ as follows:
\begin{itemize}
\item A basis $\omega^S_{i, j}$,\, $i\,=\,1,\,\cdots ,\,s$,\, $j\,=\, 2,\,\cdots ,\,j_i$, of one-forms
with simple poles at $p_{i, 1},\, p_{i, j}$ and residues $-1$ at $p_{i, 1}$, and $+1$ at $p_{i, j}$. Their
A-periods are normalised to zero, and their B-periods  on the $k$-th cycle are given
by a matrix $Y_{k, (i, j)}$.
 
\item A basis $\omega^S_{i,j,h }$, $i\,=\,1,\,\cdots ,\,s$, $j\,=\,1,\,\cdots,\,j_i$, $h\,=\,1,\,\cdots ,\,h_{i,j}$, of one-forms with
zero residue at the $p_{i,j}$, but with higher order poles at one of them, of the form $z^{-1-n_{i,j,h} }dz$. Again, their A-periods are normalized to zero, and their B-periods on the $k$-th cycle are given by $W_{k,(i,j,h)}$.
 \item A basis $\omega_i$,\, $i\,=\,1,\,\cdots ,\,\widetilde g$, of holomorphic differentials on
$\widetilde X$, with the A-periods $A_j(\omega_i) \,=\, \delta_{i,j}$, and matrix of B-periods
$B_j(\omega_i) \,=\, Z_{j,i}$; these span the cotangent space of $J(\widetilde X)$.

\end{itemize}
This gives a matrix of periods 
\begin{equation} \begin{pmatrix} 0&0&\bI\\Y& W& Z\end{pmatrix}\end{equation}
Note that the rows index the cycles, while the columns, the one-forms. We are omitting the cycles
around the points $p_{i,j}$, and their periods  ($\pm 2 \pi\sqrt{-1})$ for $\omega^S_{i, j}$.

The first two sets of forms, under the residue pairing, give a dual basis to  our  generators of $q^S(X)$.

We have the Abel map for $\widetilde X$:
\begin{align}\nonumber
\widetilde \Pi\ :\ \widetilde X& \ \longrightarrow\ J(\widetilde X)\\
p&\ \longmapsto\ \int_{p_0}^p (\omega_1,\,\cdots ,\,\omega_{\widetilde g}),
\end{align}
and its extension to $J(X)$
\begin{align} \nonumber  \Pi\ :\ \widetilde X-\{p_{i,j}\}&\ \longrightarrow\ J(  X)\\
p&\ \longmapsto\ \left(\exp\left(\int_{p_0}^p( \omega^S_{i, j})\right),\, \int_{p_0}^p(\omega^S_{i,j,h}),\,
\int_{p_0}^p (\omega_1,\cdots ,\omega_{\widetilde g})\right).
\end{align}

We recall the periodicity relations for these forms obtained by integrating one of the forms times the 
primitive of another around the usual cutting open of the Riemann surface into a $4g$-gon in the standard 
way, with boundary
$A_1,\, B_1,\, A_1^{-1},\, B_1^{-1},\,A_2,\, \cdots\,$ $B_{\widetilde g}^{-1}$ : 
\begin{align} \label{reciprocity}
  Y_{k,(i, j) }\ \ =\ \ & 2\pi \sqrt{-1} \int^{p_{i,1}}_{p_{i j}} \omega_k\\
  W_{k,(i,j,h)}\ \ =\ \ & 2\pi\sqrt{-1}   Res_{p_{i,j}} (\Pi_k \omega^{S}_{i,j,h}).\end{align}
(Here $\Pi_k$ is a primitive of $\omega_k$.)

In some cases where the curve $X$ lies in a smooth surface $Y$, the relevant forms can be obtained as
Poincar\'e residues of meromorphic $2$--forms on $Y$. If $p(x,\,y)\,= \,0$ is the equation of $X$ in
$Y$, the Poincar\'e residue of  the two-form $\Omega \,=\, \frac{f(x,y) dx\wedge dy }{p(x,y)} $ on $Y$ is the one-form on $X$
$$PR(\frac{f(x,y) dx\wedge dy }{p(x,y)})\ =\ \frac{f(x,y) dx  }{\frac{-\partial p(x,y)}{\partial y} }
\ =\ \frac{f(x,y) dy  }{\frac{ \partial p(x,y)}{\partial x} }$$
(``PR'' stands for Poincar\'e residue).
When the curve $X$ is singular, this can have poles at its singularities. If $h$ is a function on $X$,
the sum of the residues of $h\cdot PR(\Omega)$ on $\widetilde X$ is the integral on a one--cycle $\gamma$
given by the intersection of $X$ with three-spheres $S^3_i$ around the singular points $x_i$. This in turn
is the integral of $h\cdot \Omega$ on the boundary $\Gamma$ of a tubular neighbourhood $T$ of $\gamma$
in the union of the spheres $S^3_i$. But this in fact is the integral of $d(h\cdot \Omega) \,=\, 0$
on $(\sqcup_i S^3_i)- T$. Thus, over $\widetilde X$, taking the residue 
of $H\cdot PR(\Omega)$ gives a well defined result (as it annihilates lifts from $X$) on the
equivalence class $\widehat H$ of $H$ in the Lie algebra ${\rm Lie}(Q^S(X))$ of $Q^s(X)$ (see
\eqref{e1}).

\section{Theta functions}

On the Jacobian variety of the smooth $\widetilde X$, we have the theta function defined on its
universal cover $\bC^{\widetilde g}$ by the period relations for $z\,\in\, \bC^{\widetilde g}$:
\begin{itemize} 
\item $\theta(z +e_a) \,=\,\theta(z)$\,\, (A-periods; $\{e_a\}_{i=1}^{\widetilde g}$\, are
the standard basis),

\item $\theta(z+ Z_a)\,=\, \theta(z) \exp(-2\pi\sqrt{-1} z_a -\pi\sqrt{-1} Z_{aa})$\,\, (B-periods).
\end{itemize}

The theta function represents a section of an ample line bundle of positive degree over the
Jacobian which is defined by the above automorphy factors. As translations act non-trivially on the Picard variety of the Jacobian in that degree, it also defines in some sense a universal section, under the translation action.

Under the Abel map 
\begin{align} \nonumber \widetilde \Pi\ :\ \widetilde X\ &\longrightarrow\ J(\widetilde X)\\
p\ &\longmapsto\ \int_{p_0}^p (\omega_1,\,\cdots ,\,\omega_{\widetilde g}),
\end{align}
pulling back $\theta$, we are also obtaining a section over $\widetilde X$ of the pulled back bundle, which
has degree $\widetilde g$. Also, acting by translation of the embedded $\widetilde X$, we are getting
a ``universal section" for all line bundles on $\widetilde X$ of degree $\widetilde g$, which is non-zero
as long as the translated line bundle has a unique section up to scale, i.e., the translated
line bundle is generic. If not, the pull-back vanishes identically. Thus, we have Riemann's classical theorem (see, e.g. \cite[p.~336]{GH}):

\begin{proposition} 
Let $\theta_\lambda(z)\ =\ \theta(z-\lambda)$ with $\lambda \,\in\, \bC^{\widetilde g}$. Let the pull-back
of $\theta_\lambda$ under the Abel map vanish at $z_i(\lambda)$, $i \,=\,1,\,\cdots ,\, \widetilde g$,
without being uniformly zero. Then, for a suitable constant $\kappa$, 
$$\sum_{i=1}^{\widetilde g} \widetilde\Pi(z_i(\lambda) ) +\kappa\ =\ \lambda .$$
\end{proposition}

In other words, since the equivalence class of a divisor (the line bundle) is defined by its image under 
the Abel map, to obtain the divisor corresponding to $-\kappa+\lambda$, we look at the vanishing of 
$\theta_\lambda( \widetilde \Pi(z))$, and this $\theta_\lambda(\widetilde \Pi(z))$ is a section of the line 
bundle corresponding to the above mentioned divisor, with explicit automorphy factors. It is in this
sense that we can 
``invert the Abel map''.  It is this property, of providing a universal section by translation, that we 
hope to extend to singular curves.

The proof involves cutting the Riemann surface open into a $4g$--gon in the standard way and integrating 
$d\log (\theta_\lambda)$ against primitives of holomorphic one-forms around the boundary, which will give us 
$\lambda$, and a variety of constant terms which we absorb into the Riemann constant $\kappa$. We then 
compare this with the sum of residues at the $z_i(\lambda)$, which gives us the value of the Abel map. We 
will follow this pattern here, in more generality.

The extra variables in the Jacobian of $X$, however, do not represent so much new line bundles on the 
desingularisation, but rather the possible (equivalence classes of) subsheaves of line bundles on 
$\widetilde X$ representing elements of $Q^S(X)$ (see \eqref{e1}, \eqref{e2})
that will give the various line bundles on $X$. As such, 
we will think of $\theta(\Pi(z) -\lambda)$ as defining global sections of the line bundle, which generate 
the subsheaves  that we want at the preimage of the singular points. In more differential geometric terms, 
the pull-back of $d\log (\theta_\lambda)$ defines a flat connection (with poles, at the zeroes of the 
theta-function), whose covariant constant sections generate the $\pO_X$ modules we want under the lifted 
action of $\pO_X$; again, this is pertinent only at the inverse images of the singular points. Notice that 
as we are really interested in the values at the preimages of the singular points, and these are mapped to 
infinity via the Abel map, we will have to compactify our Jacobian, generally replacing our factors $\bC$ 
and $\bC^*$ in the generalized Jacobian by $\bP^1$ and thinking of our theta-function as a section of some 
line bundle (basically ${\pO}_{\bP^1}(1)$ on each $\bP^1$ factor) whose restriction to the fibers of $\bP^1$ is 
non-trivial, so that after passing to the trivialisation at infinity, the pull-back of $\theta$ takes 
finite values at the poles. We will be thinking of $\bP^1$ as being covered by two open
subsets $U_0\,=\, \{z\,\neq\, \infty\}$ (and thus containing our factors $\bC$ and $\bC^*$) and $U_1
\,=\, \{z\,\neq\, 0\}$. The transition 
function for sections of $\pO_{\bP^1}(1)$ from $U_0$ to $U_1$ will always be $z^{-1}$.

To illustrate these generalizations of $\theta$, we consider first the case when the desingularization is 
rational.

\section{Some simple cases}

\subsection {Nodal cubic}
 
Let us first consider the case of a nodal cubic $C$, parametrized after desingularizing to $\bP^1$ in such 
a way that $0$ and $\infty$ are mapped to the same point. We cover $\bP^1$ by two standard open subsets $U_0
\,=\, \{z\,\neq\, 0\}$ and $U_1\,=\, \{z\,\neq\, \infty\}$. There is a basic differential form (given for example by Poincar\'e
residues of $C$ sitting 
inside $\bP^2$),   given on $\bP^1$ by $dz/z$, with poles at 
zero and at infinity with residues $1$ and $-1$ respectively.
The generalised Jacobian here is $\bC^*$; it
corresponds to the different ways of glueing the trivial line bundle to 
itself, identifying the fibre at zero with the fibre at infinity.
The Abel map is $\Pi(z) \,=\,\exp (\int_1^zt^{-1}dt)) \,=\,\exp (\log( z)) \,=\, z$.

We choose $\Pi(z)\,=\, 1$ to correspond to the trivial line
bundle, which then should be the point where $\theta$ vanishes. Our aim is
to consider $\theta$ as a covariant constant section (or $d\log (\theta)$ as a flat connection), but
on a line bundle with non-zero degree, here $\pO(1)$; as such $\theta$ is a section with a single zero, and $d\log (\theta)$ will have a pole, with residue 1. 
Identifying the flat section at zero and infinity should give the line bundle. Setting $\theta = \xi-1$ in the
trivialization on $U_0$,  gives $\frac{\xi-1}{\xi}$ in the trivialization on $U_1$,
so that the point $-1$ (the value at $\zeta \,=\, 0$) in the fibre over $0$ gets identified
with the point $1$ (the value at $\zeta \,=\, \infty$) in the fiber over $\infty$. Displacing the
section $\xi-1$ by the automorphism $\xi \,\longmapsto\, a\xi$ gives the section $a\xi-1$, which now
has value $-1$ over $0$, and $a$ over infinity; in other words it gives us another point in
the generalised Jacobian $\bC^*$. 
 
Thus, the translates (or rather their shifts under automorphisms of $\bC^*$) of the image do give 
various line bundles on $C$.
 
\subsection{Cuspidal cubic}
 
The next example would be the cuspidal curve $C'$ defined by $x^3 \,=\, y^2$. It is desingularized by
$x \,=\, z^2,\, y \,=\, z^3$. 
The Poincar\'e residue of $\frac{dx\wedge dy}{y^2-x^3}$ is $\frac {-dy}{3 x^2} \,=\, \frac {-3z^2 dz}{3 z^4}\,= 
\, \frac{-dz}{z^2}$. Our form thus has a single double pole on $\bP^1$, and its Abel map is $z\,\longmapsto \int_{z_0}^z\frac{-dt}{t^2} \,=
\,(\frac{1}{z}- \frac{1}{z_0})$. Choosing the base point at infinity in $C'$ for the integral gives 
$z\,\longmapsto\, \zeta \,=\, \frac{1}{z}$ as the Abel map.

The generalised Jacobian is $\bC$. Indeed, ambient functions on $C'$, lifted to $\bP^1$, have no $z$--term: 
$f\,=\, a_0 +a_2z^2 + a_3z^3+\ldots$, so that each value of $a_1/a_0$ in the series $a_0 + a_1z+a_2z^2 + a_3z^3+\ldots$ gives a different $\pO_{C'}$ module.

Our flat section in the $0$--trivialisation (theta function) is simply $\zeta$,  and so $1$ in the $U_1$ 
trivialisation, as we are again in the bundle ${\pO}_{\bP^1}(1)$. It is this flat section (in the 
$U_1$--trivialization) on a neighbourhood of $\zeta \,=\, \infty$, hence $z\,=\,0$, that determines our
${\pO}_{\bP^1}$ module; here it is simply the constant function $1$. Now shifting by
$b$ gives 
$\frac{1}{z}-b$ in the $U_1$ trivialisation and so  $1-bz$ in the $U_0$ trivialisation, so a different 
element $a_1/a_0 \,=\, -b$ of the generalized Jacobian.

\subsection{A higher order  singularity on a rational curve}

We next consider the curve $C''$  given by $x^5 \,=\, y^2$. This is desingularized by $x \,=\, z^2,\, y
\,=\, z^5$. The Poincar\'e residue of $\frac{ x^n dx\wedge dy}{x^5-y^2}$ is $$\frac {x^ndy}{ 5x^4}
\ =\ \frac {5z^{4+2n} dz}{5z^8}.$$
This is singular for $n\,=\,0,\, 1$, giving forms $\frac{dz}{z^4},\, \frac{dz}{z^2}$. 

A generalised Jacobian is isomorphic to $\bC^2$, and an Abel Jacobi map is the following:
$$\Pi\ :\ z\ \longmapsto\ (\zeta_1(z),\, \zeta_2(z))\ =\ (\frac{-1}{3z^3},\, \frac{-1}{z}).$$
Thus, compactifying $\bC^2$ to $\bP^1\times \bP^1$, we have that the Abel-Jacobi map in coordinates
on $\bP^1\times \bP^1$ near infinity is given by $t_1\,=\, -3z^3$,\, $t_2\,=\, -z$.

On $\bP^1\times\bP^1$ we have the bundle ${\pO}_{\bP^1\times\bP^1}(1, 1)$. This has a standard section
$\theta(\zeta_1,\, \zeta_2)\,=\, \zeta_1\zeta_2$ which on a neighbourhood of infinity is given simply
by $1$. The translated section $\theta_{\alpha, \beta}$ is then  
$(\zeta_1-\alpha)(\zeta_2-\beta)$ in the trivialisation on $\bC^2$, and $(\zeta_1-\alpha)(\zeta_2-\beta)/\zeta_1\zeta_2$ at infinity.

The pull-back $\Pi^*(\theta)$ of $\theta$ near $z\,=\,0$ (in the $x,\,y\,=\,\infty$ trivialisation) is given simply by $1$. That of $\theta_{b_1, b_2}$ is then 
$$\Pi^*(\theta_{b_1, b_2})\ =\ \frac {(\frac {-1}{3z^3}-b_1)( \frac{-1}{z}-b_2)}{(\frac {-1}{3z^3})(
\frac{-1}{z})}
$$
$$
=\ (1+b_1 (3z^3))(1+b_2 z) = 1+ b_2 z + 3b_1 z^3 + 3b_1 b_2z^4.$$
Functions from $C''$ on the desingularization are spanned by $1,\, z^2,\, z^4,\, z^5,\, z^6,\,\cdots$. The
pull-backs  $\Pi^* \theta_{b_1, b_2}$ represent 
the Jacobian of $C''$ by $$1+ b_2z + 3b_1 z^3 + 3b_1 b_2z^4\,\ \equiv\,\ 1+  b_2 z + 3b_1 z^3:$$ these are
not pullbacks from the singular curve, but generate distinct $\pO_{C''}$--modules for each $b_1,\,b_2$. 

\subsection{Curves with $Q^S(X)$ of dimension one: one node or one cusp}

\subsubsection{A node} We begin with a simple case, of $X$ possessing one node $x$, with there being
two points $p_1,\, p_2$ in the desingularisation $\widetilde X$ mapped to $X$. One then adds to a
normalised basis $\omega_i$,\, $i\,=\,1,\,\cdots ,\, \widetilde g$, of holomorphic differentials on $X$,
a form $\omega^S$ possessing simple poles at the points $p_i$ with residue $-1$ at $p_1$
and residue $ 1$ at $p_2$, whose A-periods are $0$, and B-periods are $Y_i$,\, $i\,=\, 1,\,
\cdots,\,\widetilde{g}$. This gives for the period matrix on $\widetilde X$ a $(2 \widetilde{g})\times
(1+  \widetilde{g})$ matrix 
\begin{equation}\label{em}
M\ =\ \begin{pmatrix}   0 &   \bI\\Y& Z\end{pmatrix},
\end{equation}
though one should not forget the non-trivial residues in the poles of $\omega^S$.

The Jacobian of $X$ is then the quotient of $\bC\oplus \bC^{\widetilde g}$ by translation by the rows of
the matrix in \eqref{em}, giving (after an exponential in the first factor)
an extension by $\bC^*$ of the Jacobian $J(\widetilde X)$. This will be compactified to a $\bP^1$-bundle
over $\pO(1)$. We   have an Abel map 
\begin{equation} \label{Abel1}\Pi(p)\ =\ ( \exp(\xi(p)),\, z(p)) =\ ( \exp(\xi(p)),\, z_1(p), \, z_2(p),\,\cdots ,\, z_{\widetilde g}(p) )\end{equation}
with
$$(\xi(p),\, z_1(p), \, z_2(p),\,\cdots ,\, z_{\widetilde g}(p) )\ =\  (\int_{p_0}^{p}  \omega^S,\, \int_{p_0}^{p}( \omega_1,\,\cdots ,\, \omega_{\widetilde g})).$$

A theta function was in fact given in the nineteenth century by Clebsch \cite{Cl}, in terms of the 
theta function $\widetilde\theta$ associated to the desingularisation $\widetilde X$. Our normalisations will 
be slightly different. We set, for $(\xi ,\, z)\,\in\, \bC\times \bC^{\widetilde g}$
\begin{equation}
\theta (\xi,\, z)\ =\   \exp({\xi}) \widetilde{\theta}(z+\nu) + \widetilde{\theta} (z ).  
\end{equation}
Here $\nu_\alpha \,=\, \int_{p_1}^{p_2} \omega_\alpha$, with the integrals taken as usual on a
cutting open of $\widetilde X$ into a $4\widetilde g$--gon. Note that by our reciprocity
relations (\ref{reciprocity}), $$Y_\alpha\ =\ 2\pi\sqrt{-1}\nu_\alpha .$$

This function is invariant under shifting in $z$ by an $A$-period; for the $B$-periods, we have
(using the reciprocity relations)
\begin{align*}\label{B-period}
\theta (\xi + Y_\alpha, z + Z_\alpha)\ =&\ \exp ({\xi+Y_\alpha}) \widetilde{\theta} (z + \nu +Z_\alpha ) +
\widetilde{\theta}(z + Z_\alpha)\\
=&\ \exp ({\xi}) \widetilde{\theta} (z + \nu )\exp (Y_\alpha -2\pi\sqrt{-1} (z_\alpha +\nu_\alpha) -
\pi\sqrt{-1}Z_{\alpha\alpha})\\
&\qquad + \widetilde{\theta} (z )\exp(-2\pi\sqrt{-1}z_\alpha  -\pi
\sqrt{-1} Z_{\alpha\alpha})\\
=&\ (\exp({\xi}) \widetilde{\theta} (z + \nu ) +\widetilde{\theta} (z ))
\exp(-2\pi \sqrt{-1} z_\alpha -\pi\sqrt{-1} Z_{\alpha\alpha})\ \\
=&\ \theta (\xi , z) 
\exp(-2\pi\sqrt{-1} z_\alpha -\pi\sqrt{-1} Z_{\alpha\alpha}).
\end{align*}
That is the same quasi-periodicity relation as for ordinary $\theta$ functions.
Note that because of the periods $Y_\alpha$, the Jacobian is not in any obvious way a product, and so the restriction of the theta functions to the fiber $\bC^*$ vary with $z$. 

Shifting the image of the curve under the Abel map by $(a,\, \lambda) \,\in\, \bC^*\times \bC^{\widetilde g}$,  
we have 
\begin{align*}\theta_{a,\lambda}(\xi,\, z)\,\ =\,\ & a^{-1} \exp (\xi )\widetilde{\theta}(z-\lambda+\nu)+\widetilde{\theta} (z-\lambda )\\
\ =\ &\exp ({\xi -\log(a)})\widetilde{\theta}(z-\lambda+\nu)+ \widetilde{\theta} (z-\lambda )
\end{align*}
periodic under a shift by an $A$-period, and with 
\begin{equation}\label{B-shift} \theta_{a,\lambda}(\xi +Y_\alpha,\, z + Z_\alpha)\ =\  \theta_{a,\lambda}(\xi  , z )
\exp (-2\pi\sqrt{-1} (z_\alpha -\lambda_\alpha) -\pi\sqrt{-1} Z_{\alpha\alpha})\end{equation}
for the B-periods. With as above $(\xi,\, z)(p)\,=\, (\xi,\, z_1,\, \cdots ,\,z_{\widetilde g})(p)$ the
primitives   of the forms $(\omega^S,\, \omega_1,\,
\cdots ,\,\omega_{\widetilde g})$ we  compute the integral of the product of
$$\frac{1}{2\pi\sqrt{-1}}d\log \theta_{a,\lambda}((\xi, z)(p))\ =\ \frac{1}{2\pi\sqrt{-1}}d\log
\left(  a^{-1}\exp({ \xi}) \widetilde \theta( z -\lambda+ \nu)+ \widetilde{\theta} (z-\lambda )\right)$$
with the primitives $ (\xi,\, z_1,\,\cdots ,\, z_{\widetilde g})$. We compute this around the boundary
of a standard cutting open of $\widetilde X$ into
\begin{itemize} 
\item a $4\widetilde g$--gon with boundary given in terms of a symplectic homology basis  by $$A_1,\,B_1,\,
A_1^{-1},\, B_1^{-1},\, A_2,\, \cdots ,\, B_{\widetilde g}^{-1},$$ starting and ending at a point $p_0$, 

\item to  which one adds disjoint loops $C_i, i=1, 2$ starting at $p_0$ on the boundary, and going around the $p_i$, and 

\item  loops $D_j$ starting at $p_0$ and going around the zeros $q_j$ of $\theta_{a,\lambda}(\Pi(p))$,
with all $C_i$, $D_j$ disjoint except at $p_0$. There will be $g = g(X) = g(\widetilde X) +1$ of these zeroes.
 \end{itemize}
  We have, for the integral of $\frac{1}{2\pi\sqrt{-1}} d\log \theta_{a,\lambda}((\xi,\,z)(p))\times
(\xi,\, z_1,\,\cdots ,\,z_{\widetilde g})$ the following:
\begin{itemize} 
\item The integrands on $A_\alpha,\,  A_\alpha^{-1}$ differ by their change along the  cycle $B_\alpha$.
This gives, from (\ref{B-shift}) and the periods of our forms, that the integral on $A_\alpha + A_\alpha^{-1}$ is on $A_\alpha$ that of
\begin{align*}
\frac{1}{2\pi\sqrt{-1}}\Big[d\log \theta(\xi -\log a,\, z-\lambda )\times &(\xi,\, z_1,\,\cdots ,\,
z_{\widetilde g})\\
-d\log \{\theta(\xi -\log a, z-\lambda)\exp (\pi\sqrt{-1}(-2z_\alpha -Z_{\alpha\alpha})\}
\times &(\xi +Y_\alpha,\, z_1+Z_{\alpha,1},\,\cdots ,\, z_{\widetilde g}+Z_{\alpha, \widetilde g})\Big]\end{align*}
giving $ dz_\alpha \times  (\xi,\, z_1,\,\cdots ,\,z_{\widetilde g})$ plus 
\begin{equation*}-\frac{1}{2\pi\sqrt{-1}}\big[d\log (\theta(\xi -\log a,\, z-\lambda)
\exp( \pi \sqrt{-1} (-2z_\alpha -Z_{\alpha\alpha}))\big] \times  ( Y_\alpha,\, Z_{\alpha,1},\,
\cdots ,\, Z_{\alpha, \widetilde g}).
\end{equation*}
The integral of the first term does not depend on the shifts $a,\, \lambda$ and so it is a constant,
the second is the integral of the differential of a log (times constants), and as the theta functions
are well defined on the A-cycles, this gives us multiples of $2\pi\sqrt{-1}$ times those constants.
In short, the contribution of $A_\alpha + A_\alpha^{-1}$ to the contour integral is a constant, which we incorporate into the Riemann constant.

\item Now consider the cycles  $B_\alpha,\, B_\alpha^{-1}$; here the values of the integrand over these
cycles on the cut-open curve differ by their change along $A_\alpha^{-1}$, which is zero for the theta function, and zero for all the primitives except $z_\alpha$, where it is minus one.  We have 
$$
\frac{1}{2\pi\sqrt{-1}}[d\log\theta(\xi -\log a,\, z-\lambda )\times (\xi,\, z_1,\,\cdots ,\, z_{\widetilde g})
$$
$$
-d\log\theta(\xi -\log a,\, z-\lambda)\times   (\xi ,\, z_1 ,\, \cdots ,\, z_\alpha-1,\,\cdots ,\,
z_{\widetilde g} ) ]
$$
$$
=\ \frac{1}{2\pi\sqrt{-1}} [ d\log [\theta(\xi -\log a,\, z-\lambda)  ] \times ( 0,\,\cdots ,\, 0,\,
1,\,0,\, \cdots ,\,0)
$$
which when integrated, using the quasi-periodicity relation along the $B$-cycles, gives zero for $\xi,\,
z_{i, i\neq \alpha}$, and a constant plus $\lambda_\alpha$ for the primitive $z_\alpha$. 
Note that calculations for the $A$-, $B$-cycles only use the $B$--quasi-periodicity relations for $\theta$, the fact that
$\theta$ is periodic in the $A$--cycles, and that the $A$--periods of the singular forms are zero; so when we repeat the calculations for the $A$ and $B$ cycles later for more complicated  
singularities, we will get the same contributions, if our theta-functions still satisfy the same relations.

\item  In turn, for  the integral on the two $C_i$, (on the sum of the two loops, $\xi$ is well defined)
of 
\begin{align*}\frac{1}{2\pi\sqrt{-1}}&d   \log \theta_{a,\lambda}((\xi,\, z)(p))\times (\xi,\, z_1,\,\cdots ,\,
z_{\widetilde g}) \\ &=\frac{1}{2\pi\sqrt{-1}}d   \log \theta ((\xi-\log(a),\, z-\lambda)(p))\times (\xi,\, z_1,\,\cdots ,\,
z_{\widetilde g})  \end{align*}
we have two singular points inside $C_1\cup C_2$, one ($p_2$) where $\exp(\xi) $ is zero ($\omega^S$ has
a single pole of residue $1$), and one ($p_1$) where $\exp(\xi) $ has a pole ($\omega^S$ has a single pole of residue $-1$).

At $p_2$, everything is finite in $  \log \theta$; integrating by parts, the $\xi$ integral  becomes the residue at $p_2$ of
$$-\log \theta ((\xi-\log(a),\, z-\lambda)(p))\omega^S$$ 
i.e. the value of $-\log \widetilde{\theta} (z(p_2)-\lambda )$. The $z_i$ integrals give zero.

At $p_1$, we can change trivializations, dividing by $\exp(\xi)$, which changes the contour integral by a constant, which as before we move into the Riemann constant; so now 
$$\theta_{a,\lambda} = \frac{\exp ({\xi -\log(a)})\widetilde{\theta}(z-\lambda+\nu)+
\widetilde{\theta} (z-\lambda )}{\exp(\xi)}$$
and again $d   \log \theta$ becomes finite. This tells us that the integrals with $z_1,\,\cdots ,\,
z_{\widetilde g}$ on  $C_2$ give zero, up to constants. For the integral with $\xi$, we now want the residue of 

$$- \log \frac{\exp ({\xi -\log(a)})\widetilde{\theta}(z-\lambda+\nu)+
\widetilde{\theta} (z-\lambda )}{\exp(\xi)}\omega^S$$ 

which as $\omega^S$ has residue $-1$, gives $- \log(a) +\log  (\widetilde{\theta}(z(p_1) +\nu -\lambda)) $.
The integral against $\xi$ on the whole contour then sums to 
$$- \log(a) +\log \frac{ (\widetilde{\theta}(z(p_1) +\nu -\lambda))}{\widetilde{\theta} (z(p_2)-\lambda )} $$

\item The integral on the $D_j$ contribute $  n_j (\xi,\, z_1,\,\cdots ,\,z_{\widetilde g})$, where
$n_j$ is the multiplicity of the zero of $\theta_{a,\,\lambda}((\xi,z)(p))$.\\
\end{itemize}
 
Summarising, as the integrals on the $A,B$ contour is  the same as on the $C +D$ contour ( with the standard orientations), the following proposition holds.

\begin{proposition}
If $q_j(a,\, \lambda)$ are the zeroes of $\theta_{a,\lambda}((\xi,\,z)(p))$, repeated according to multiplicity,
then, setting 
$$\sum_{j=1}^{\widetilde g +1} (\exp(\xi(q_j(a,\, \lambda))),\, z(q_j(a, \lambda)))  
\ =\ (C   \frac {\widetilde{\theta} (z(p_2)-\lambda )} { (\widetilde{\theta}(z(p_1) +\nu -\lambda))} a,\,  \kappa +\lambda) ,$$
where $C, \kappa $ are now   generalised Riemann constants.
\end{proposition}

Note that the type of calculation involved is in essence one we are already seeing in the smooth case 
involving Abelian integrals. We also note that the linear shift $(a, \lambda)$ along the Jacobian (i.e the action of translation on line bundles on the Jacobian) does not turn into a linear shift of the values of the Abel map, though it is in a way the next best thing, i.e. a linear shift by $\lambda$ in the value of the Abel map of $\widetilde X$, and a linear shift by $\log(a)$ in the fiber $\bC^*$, plus a non-linear term in $\lambda$, which would allow us to invert.  The pattern, unfortunately, becomes even more complicated when there are multiple singularities. 

\subsubsection{A cusp.}

Now let us take a curve $X$ with a single simple cusp singularity.
If $\widetilde X$ is its desingularisation, with the preimage of the singular point $x$ being a
single point $q$, then there is the following short exact sequence:
$$0\,\longrightarrow\, \bC\,\longrightarrow\, J(X)\,\longrightarrow\, J(\widetilde X)
\,\longrightarrow\, 0.$$
Again, we  add to the holomorphic differentials $\omega_i$ a single singular differential $\omega^S$
with a double pole and no residue at $q$. The matrix of periods is again a
$(\widetilde g +\widetilde g)\times(1+ \widetilde g )$ matrix 
$$ N\ =\ \begin{pmatrix}   0 &\bI \\  W& Z\end{pmatrix}.$$
We have an Abel map $\Pi\, :\, \widetilde{X}\,\longrightarrow\, J(X)$
$$ (\zeta,\, z)(p) \,=\, (\zeta(p),\, z_1(p),\, z_2(p),\,\cdots ,\,z_{\widetilde g}(p))
\,=\, (\int_{p_0}^{p}  \omega^S,\, \int_{p_0}^{p}( \omega_1,\,\cdots ,\,\omega_{\widetilde g})).$$
We want a function $\theta(\zeta,\, z)$ that is periodic in shifting $z$ by an element of the integer
lattice (an A-period). Also, shifting $(\zeta,\, z)$ by a B-period it should translate by
$\exp(-2\pi\sqrt{-1} z_\alpha -\pi \sqrt{-1} Z_{\alpha\alpha})$, and it should be affine linear in $\zeta$.
Writing $\theta(z) \,=\, \phi(z) + \psi(z)\zeta$, we want the usual quasi-periodicity
$$ \phi(z +Z_\alpha) + \psi(z+Z_\alpha)(\zeta + W_\alpha)
\,=\, (\phi(z) + \psi(z)\zeta)\exp (-2\pi \sqrt{-1} z_\alpha -\pi\sqrt{-1} Z_{\alpha\alpha}).$$
Now remark that if $\bD\,=\, \sum_\mu \frac{W_\mu}{2\pi\sqrt{-1}}\frac{\partial}{\partial z_\mu}$, we have 
$$
\bD(\widetilde{\theta}(z+Z_\alpha) )\ =\ (\bD(\widetilde{\theta}(z)) - W_\alpha \widetilde\theta(z))
\exp(-2\pi \sqrt{-1}z_\alpha - \pi\sqrt{-1} Z_{\alpha\alpha}).$$
Thus setting $\psi(z) \,=\, \widetilde{\theta}(z),\, \phi(z)\,=\,\bD(\widetilde{\theta}(z))$ gives us
the relations we want, and so we set
$$
\theta(\zeta,\, z)\, = \,\bD(\widetilde{\theta}(z)) + \widetilde{\theta}(z)\zeta .
$$
Again, integrating the product of 
$$
\frac{1}{2\pi\sqrt{-1}} d\log \theta_{b,\lambda}((\zeta,\, z)(p))\, =\,
\frac{1}{2\pi\sqrt{-1}} d\log \theta(\zeta(p) -b,\, z(p)-\lambda)
$$
$$
= \ \frac{1}{2\pi\sqrt{-1}} d\log (\bD \widetilde{\theta}(z(p)-\lambda) + 
\widetilde{\theta}(z(p)-\lambda)(\zeta(p)-b))
$$
with the primitives $\zeta,\, z_1,\,\cdots ,\, z_{\widetilde g}$ of $\omega^S,\, \omega_1,\, ...,\,  \omega_{\widetilde g}$  around the cut-open surface (so cutting on the standard $A,B$ cycles, with a cycle $C$ surrounding the pole, and a cycle $D$ surrounding the zeroes of $\theta_{b,\lambda}((\zeta,\, z)(p))$,
we find that the $A$-cycles and the $B$-cycles contribute as before, while now integrating around the cycle
$C$ that goes around the singularity, we are taking the residue of 
$$d\log[ \bD\widetilde{\theta}(z(p)-\lambda) +  \widetilde{\theta}( z(p)-\lambda)(\zeta(p)-b)]\times
(\zeta,\, z_1,\, \cdots ,\, z_{\widetilde g}),$$
Again, passing to the trivialisation near infinity in the $\zeta$-variable, i.e., dividing $\theta_{b,\lambda}((\zeta,\, z)(p))$ by $\zeta(p)$ just changes the contour integral on $C$ by constants; we are then integrating the product of something finite with $(\zeta,\, z_1,\, \cdots ,\, z_{\widetilde g})$; the integrals of the product with $ z_1,\, \cdots ,\, z_{\widetilde g}$ give zero. On the other hand for the  product with $\zeta$, we have the integral of 
$$\frac{1}{2\pi\sqrt{-1}}  d\log[  \widetilde{\theta}( z(p)-\lambda) + \frac{ \bD\widetilde{\theta}(z(p)-\lambda) -  \widetilde{\theta}( z(p)-\lambda)b)}{\zeta}]\times \zeta  $$
with generically a finite non-zero value at $q$. This gives an integrand
$$\frac{1}{2\pi\sqrt{-1}}  [\frac { d\widetilde{\theta}( z(p)-\lambda)  +  ( \bD\widetilde{\theta}(z(p)-\lambda) - b\widetilde{\theta}( z(p)-\lambda) d(\zeta^{-1}) } {\widetilde{\theta}( z(p)-\lambda)}+ O(\zeta^{-1}) ]\times \zeta$$
and so a residue of 
$$ \frac { -\frac{d}{dp}\widetilde{\theta}( z(q)-\lambda)  +  ( \bD\widetilde{\theta}(z(q)-\lambda) } {\widetilde{\theta}( z(q)-\lambda)} - b$$ 

   The $D$ cycle again gives the divisor of the zeroes of the
theta-function, and so we have the following:

\begin{proposition}
If $q_j(a,\, \lambda)$ are the zeroes of $\theta_{b,\lambda}((\zeta, \,z)(p))$, repeated according to multiplicity,
$$\sum_{j=1}^{\widetilde g +1} (\zeta(q_j(b, \,\lambda)),\, z(q_j(b,\, \lambda)))+ \kappa\
=\ (b + C(\lambda),\, \lambda),$$
where $\kappa $ is again a generalized Riemann constant, and 
$C(\lambda) =  \frac { -\frac{d}{dp}\widetilde{\theta}( z(q)-\lambda)  +  ( \bD\widetilde{\theta}(z(q)-\lambda) } {\widetilde{\theta}( z(q)-\lambda)} $
\end{proposition}

\section{More singular points} 

\subsection{Pulling back from $(\bP^1)^n$; curves with rational desingularizations}

We want to extend the idea of the theta function as generating a universal section under the action of the 
translation group on a smooth curve in the Jacobian of the curve, to also generating a universal 
trivialisation at the pre-images of the singular points under the same translation action, now extended to 
the generalised Jacobian. We will start with the case of a rational desingularisation: $\widetilde X \,=
\, \bP^1$. 
The generalised Jacobian $Q^S(X)$ is then $(\bC^*)^{s_1} \times \bC^{s_2}$ as a variety, with a compactification 
$(\bP^1)^{s_1+s_2}$, and we are looking at the line bundle $\pO_{(\bP^1)^n}(1,1,\cdots ,1)$ over this, with the standard  trivialisations of the 
line bundle ${\pO}_{(\bP^1)^{s_1+s_2}}(1,1,\cdots ,1)$ on the divisors added at infinity to compactify. For the
factors $\bC^*$ in 
the generalised Jacobian, we are adding points over $0,\,\infty$, and for the factors $\bC$, we are 
adding ($2$-times) $\infty$. The group $(\bC^*)^{s_1} \times \bC^{s_2}$ acts transitively on $Q^S(X)$.

For $ X$ a singular curve with $\widetilde{X} \,=\, \bP^1$ as a desingularisation, we have a basis of forms 
$\omega_j^S$ spanning via residues the dual of the tangent space to $Q^S(X)\,=\, J(X)$
(see \eqref{e1}). The first $s_1$ 
forms have two simple poles with residues $\pm 1$, and so the exponential of their integrals take values in 
$\bC^*$, with limits at the singular points either $0$ or $\infty$; the remaining $s_2$ forms have higher order poles with no residues, and 
  have integrals that take values in $\bC$, with limit $\infty$. Recall that we are 
looking at sections of $\pO_{\bP^1}(1)$, and must therefore look at the values at infinity in the appropriate 
trivialisations.

As in our first example, for $\bC^*\,\subset\, \bP^1$, we have a standard section of $\pO_{\bP^1}(1)$ on $\bP^1$ given on the 
standard trivialisation of $\pO_{\bP^1}(1)$ on $\exp(\xi)\neq\infty$ by $\exp(\xi) -1$, and the action of $a\in \bC^*$ on this is
given by 
$ a^{-1}\exp(\xi) - 1$. On the trivialisation near $\exp(\xi)\, =\, \infty$, the standard action is then 
$a^{-1} - 1/\xi$. The action on the trivialisation at $\exp(\xi)\, =\, 0$ is then trivial, and the action on the 
trivialisation at $\exp(\xi)\, =\, \infty$ is by $a^{-1}$.

For $\bC\,\subset\, \bP^1$, our standard section is $\zeta$; the translation action is by $b(\zeta) \,=\, \zeta -b$ and so $ 
\frac{1}{\zeta} (\zeta -b)\, =\, 1-b\zeta^{-1}$ at infinity. Thus the action on the trivialisation on the first formal 
neighbourhood of infinity is by $1\,\longmapsto\, 1-b\zeta^{-1}$.

This is in one variable; when we are looking at the trivialisations of ${\pO}_{(\bP^1)^{s_1+s_2}}(1,1,\cdots ,1)$, we
are looking at
the action of $(a_1,\, a_2,\, \cdots , a_{s_1},\, b_1,\,\cdots ,\,b_{s_2})$ on the standard product section
$$\theta(\xi,\zeta)\,=\,\prod_{i=1}^{s_1} (\exp(\xi_i)-1) \prod_{i=1}^{s_2}\zeta_i$$
in the standard trivialisation:
$$\theta_{a,b}(\xi,\zeta)\,=\,\prod_{i=1}^{s_1} (a_i^{-1}\exp(\xi_i)-1) \prod_{i=1}^{s_2}(\zeta_i-b_i)$$
and then taking the shift that this action on the section produces in the various    trivialisations of
${\pO}_{(\bP^1)^{s_1+s_2}}(1,1,\cdots ,1)$.    Thus, to describe
the action of $(a_1,\, a_2,\, \cdots , a_{s_1},\, b_1,\,\cdots ,\,b_{s_2})$ on these trivialisations, we first
partition the 
   coordinates $(\xi_1,\,\cdots ,\,\xi_{s_1},\, \zeta_1,\,\cdots ,\,\zeta_{s_2}) $, into sets , where they are said to belong to $U_0$ if
they (i.e., $\exp(\xi_i), \zeta_i$) are not near infinity, and to belong to $U_1$ if they are near infinity, (which of course involves arbitrary choices
when the coordinates are neither zero nor infinity, and is the choice of which standard trivialisations
we are using) and take the appropriate trivialisations of $\pO_{(\bP^1)^{s_1+s_2}}(1,1,\cdots ,1)$ there. Then, for
example, the shifted theta-function  is given by
$$ \prod_{\xi_{i}\in U_1,\,  } (a_{i}^{-1}-\exp(\xi_i)^{-1}) \prod_{\xi_i\in U_0} (a_{i}^{-1}\exp(\xi_i )- 1) \prod_{\zeta_{j} \in U_1} (1-b_{j}\zeta_{j}^{-1})\prod_{\zeta_{j} \in U_0} (\zeta_{j}-b_{j}).$$
 

Now we are interested in the action of the translation on pull-backs of $\theta$ under the Abel
map of ${\widetilde X}\,=\, \bP^1$ into $\bP^n$ given by the integrals of forms. 
Let  the index $i$ indicate which singular point $x_i$ on $X$ the function  or form is associated to; the index
$j$ indexes the different points $p_{i,j}$,\, $j\,=\,1,\,\cdots ,\, j_i$, in the preimage of $x_i$ in $\bP^1$, and the
index $h$,\, $h\,=\, 1,\, \cdots ,\,h_{i,j}$ the different higher order forms. Thus we have 
\begin{itemize}
\item Forms $\omega^S_{i, j}$,\, $i\,=\,1,\,\cdots ,\,s$,\, $j\,=\, 2,\,\cdots ,\,j_i$,
with simple poles at $p_{i,1}$ and $p_{i,j}$, with residues $1,\, -1$ respectively: $\omega^S_{i, j}(z) = ((z-p_{i,1})^{-1}- (z-p_{i,1})^{-1})dz$

\item forms  $\omega^{S}_{i,j,k}$,\, $h \,=\, 1,\,\cdots ,\,h_{i,j}$, with residues zero, and singularities at the
points $p_{i,j}$: $\omega^{S}_{i,j,h}\  =\ (z-p_{i,j})^{-1-n_{i,j,h}} dz$.
\end{itemize}
(We choose  a coordinate on $\bP^1$ so that none of the $p_{i,j}$ are infinite.)
 
 These will be dual to the Lie algebra $q^S(X)$, whose elements are represented by giving
germs of functions $f_{i,j},\, f_{i,j,h}$ at the $p_{i,j}$. We can normalise our  functions as in section 2.

Set $s_1 \,=\, \sum_i(j_i-1)$,\,
$s_2\, =\, \sum_i h_{i,j}$. The Abel map
$$\Pi\ :\ \bP^1-\{p_{ij}\}\ \longrightarrow\  (\bC^*)^{s_1} \oplus \bC^{s_2}$$ is defined by 
  $$p\ \longmapsto\ (\exp(\xi_{i,j}(p)) ,\, \zeta_{i,j,h}(p))\ =\ (\exp(\int_{p_0}^p \omega^S_{i, j}) ,
\, \int^p_{p_0}\omega^{S}_{i,j,h}).$$
These maps extend to maps from $\bP^1$ to the compactification $(\bP^1)^{s_1} \times (\bP^1)^{s_2}
\,\supset\, (\bC^*)^{s_1} \oplus \bC^{s_2}$. The component $\exp(\xi_{i,j})$\, ($j\,>\,1$) is a map of degree
one, sending $p_{i,1}$ to zero, and $p_{i,j}$ to infinity. The components $\zeta_{i,j,h}$, being the integrals of
forms with poles of order $-1-n_{i,j,h}$ (and no residue) at the $p_{i,j}$, are maps of degree $n_{i, j,h}$, and send the singular points to infinity. 

Now consider on $(\bP^1)^{s_1} \times (\bP^1)^{s_2}\,\supset\, (\bC^*)^{s_1} \oplus \bC^{s_2}$ the line bundle 
$\pO_{(\bP^1)^{s_1+s_2}}(1,\,1,\,\cdots ,\,1)$, and take 
\begin{equation}\label{def-theta-rat}
\theta( \xi_{i,j},\, \zeta_{i,j,h} )\ =\ \prod_{i,j}(\exp(\xi_{i,j}) -1) \prod_{i,j,h} \zeta_{i,j,h}\end{equation}
in the standard trivialisation on the set $U_{0,0,\cdots ,0}$ for which the $\xi_{i,j},\, \zeta_{i,j,h}$
are finite. Note that not all components of the Abel map are simultaneously infinite. 
Next consider the shift $\theta_{a,b}$ induced by $\xi_{i,j}\,\longmapsto\,a_{i,j}^{-1}\xi_{i,j}$ and
$\zeta_{i,j,h}\,\longmapsto
\,\zeta_{i,j,h}- b_{i,j,h}$. The pull-backs of $\theta_{a,b}$, or equivalently, of $\log(\theta_{a,b})$, give us
our trivializations at preimages of the singular points in $\widetilde X$, (provided they are evaluated
in the appropriate trivialisation of $\pO_{(\bP^1)^{s_1+s_2}}(1,1,\cdots ,1)$) and their  coordinates in
$q^S(X)$ in terms of the generators  $f_{i,j}, f_{i,j,h}$ are given by the quantities
$${\rm Res}_{p_{i,j}} \log (\Pi^*\theta_{a,b}) \omega^S_{i, j} + {\rm Res}_{p_{i,1}} \log (\Pi^*\theta_{a,b}) \omega^S_{i, j},\ \ \ \,
  {\rm Res}_{p_{i ,j}} \log (\Pi^*\theta_{a,b}) \omega^{S}_{i,j,h} $$
taking care about using trivialisations of $\pO_{(\bP^1)^{s_1+s_2}}(1,1,\cdots ,1)$ that are valid  near the images of
the $p_{i,j}$. Alternately, this will be the same as taking a contour integral of 
$$\frac{-1}{2\pi \sqrt{-1}}\sum_{i,j}  (d\log (\Pi^*\theta_{a,b})) \xi_{i, j} ,\ \ \ \,
\frac{-1}{2\pi \sqrt{-1}}\sum_{i,j} (d\log (\Pi^*\theta_{a,b})) \zeta_{i,j,h}$$
on suitable contours $C$, which does not contain the   zeroes of $\Pi^*\theta_{a,b}$, which, generically, we can
assume are distinct from the $p_{i,j}$. The changes of trivialisation (chosen so that $d\log (\Pi^*\theta_{a,b})$ is always finite, non-zero) introduce constants in the contour integral  that are independent
of $a,\,b$ , which we absorb into a generalised Riemann constant.  Let $D$ be the contour surrounding the zeroes of $\Pi^*\theta_{a,b}$, so that the integral on $C+D$ (homologous to a circle near infinity) is zero, again, up to a Riemann constant.
 
The contour integral   $\frac{-1}{2\pi \sqrt{-1}}\int_C d\log (\Pi^*\theta_{a,b})\xi_{i, j}$ is a sum of terms 
\begin{itemize} 
\item $Res_{p_{i',1}}  \log (a_{i',j'}^{-1}-\exp(-\xi_{i',j'} ) )\omega^S_{i, j} + Res_ {p_{i,j}} \log (a_{i',j'}^{-1} \exp(\xi_{i',j'}) -1)\omega^S_{i, j}
{\buildrel {\rm{def}}\over {=}}
M_{i',j';i,j}(a_{i',j'})$
\item $Res_{p_{i,1}}  \log (1-b_{i'j'h'}\zeta_{i',j',h'}^{-1})\omega^S_{i, j} + Res_ {p_{i,j}}  \log (1-b_{i'j'h'}\zeta_{i',j',h'}^{-1})\omega^S_{i, j}
{\buildrel {\rm{def}}\over {=}} 
N_{i',j',h';i,j}(b_{i',j',h'}) $
\end{itemize}
 In the same way, the contour integral   $\frac{-1}{2\pi \sqrt{-1}}\int_C d\log (\Pi^*\theta_{a,b})\zeta_{i, j,h}$ is a sum of terms 
\begin{itemize} 
\item $  Res_ {p_{i,j}}  -d(\log (a_{i',j'}^{-1} \exp(\xi_{i',j'} )\zeta_{i, j,h}{\buildrel {\rm{def}}\over {=}}
P_{i',j';i,j,h}(a_{i',j'})$
\item $ Res_ {p_{i,j}} -d\log (1-b_{i'j'h'}\zeta_{i',j',h'}^{-1})\zeta_{i, j,h}{\buildrel {\rm{def}}\over {=}}
Q_{i',j',h';i,j,h}(b_{i',j',h'}) $
\end{itemize}

We then have 
\begin{proposition}
 The shift by $(a,b)$ on the Jacobian $Q^S(X)$ induces via the contour integral around $C$ of the corresponding shift in the theta-function gives via  $C$ a shift in the Lie algebra $q^S(X)$ given by 
 \begin{align} \delta(f_{i,j}) =&   \sum_{i'j'} M_{i',j';i,j}(a_{i',j'}) + \sum_{i,j,h}N_{i',j',h';i,j}(b_{i'j'h'})\\
 \delta(f_{i,j,h})=&   \sum_{i'j'} P_{i',j';i,j,h}(a_{i',j'}) + \sum_{i,j,h}Q_{i',j',h';i,j,h}(b_{i'j'h'})
 \end{align}\
 \end{proposition}
 
\begin{remark}
We have here in some sense two linear structures, one on the Jacobian of $X$ and the other on 
the Lie algebra $q^S(X)$, though of as germs of functions at the $p_i$. We would have hoped, in analogy to 
theta-functions and Jacobians for smooth curves, that the correspondence given by our theta-function would be 
linear, indeed, up to constants, the identity map. This is not the case, and we have found no simple way of 
making it so. Basically, while the description of $q^S(X)$ in terms of generators $f_{i,j}$ isolates each 
singularity, the theta function mixes them all up. The functions $M,N,P,Q$ have some interesting features, 
however: for one, each one of them depends on only one coordinate, which corresponds to the first set of 
indices:$ M_{i',j';i,j}$ depends only on $a_{i',j'}$, and so on. In simple cases, for example for curves with 
a single cusp type singularity of any order, we can invert explicitly, computing the $b_h$'s from $Q_{ h';h}$ 
but in general it does not seem to be easy to do.
\end{remark}

As we have noted above, the sum of our integrals around the contours $C,\,\, D$ is zero. The integrals of 
$d\log(\theta_{a,b})(\xi,\, \zeta)$ around the zeroes of $d\log(\theta_{a,b})(\xi,\, \zeta)$ gives the values at 
these zeroes of $(\xi,\, \zeta)$; thus we have the following:

\begin{theorem} Let the desingularisation of $X$ be $\bP^1$, and  let the Abel map $\Pi\,:\, \bP^1\,
\longrightarrow\, Q^S(X)$ be given by the integrals
$\exp(\xi_{i,j})),\,\zeta_{i,j,h}$ (followed for the $\xi_{i,j}$ by exponentiation).

If $q_\mu(a,\,b)$,\, $\mu\, =\, 1,\, \cdots,\, g$ are the zeroes of $\Pi^*\theta_{a,b}((\exp(\xi_{i,j}(p)),
\,\zeta_{i,j,h}(p))$, repeated according to multiplicity,
then
\begin{align*}\sum_{j=1}^{g} \Bigl(\xi_{i,j}(q_\mu(a,\, b),\, \zeta_{i,j,h}(q_\mu(a, b))  \Bigr ) 
\ =\  &\  \Bigl( \sum_{i'j'} M_{i',j';i,j}(a_{i',j'}) + \sum_{i',j',h'}N_{i',j',h';i,j}(b_{i'j'h'}),\\
&  \sum_{i'j'} P_{i',j';i,j,h}(a_{i',j'}) + \sum_{i',j',h'}Q_{i',j',h';i,j,h}(b_{i'j'h'}) \ \Bigr ). 
\end{align*}
\end{theorem}

\section{Desingularisations  of higher genus}
 
\subsection{Arbitrary cuspidal singularities} 

We begin with a special case. We define a cuspidal singularity to be one whose inverse image in the desingularisation is a single point.
Now let us suppose that we have a curve $X$ whose only singularities are cuspidal;  the Jacobian will then be an extension 
$$0\,\longrightarrow\, \bC^N\,\longrightarrow\, J(X)\,\longrightarrow\, J(\widetilde X)\,\longrightarrow\, 0.$$  Let $p_i, i=1,..s$ denote the points in $\widetilde X$ corresponding to the singularities.
Again, the dual of the  Lie algebra of $J(X)$ will be spanned by the following forms: 
\begin{itemize}
\item A basis $\omega_i$,\, $i\,=\,1,\,\cdots ,\,\widetilde g$, of holomorphic differentials on $\widetilde X$, with
the A--periods $A_j(\omega_i) \,=\, \delta_{i,j}$, and matrix of B-periods $B_j(\omega_i) \,=\, Z_{i,j}$; these
span the cotangent space of $J(\widetilde X)$.

\item A basis $\omega^{S}_{i,h},\, i \,=\, 1,\, \cdots,\, s$,\, $h \,= \, 1,\,\cdots ,\, h_i$ of  one-forms on $\widetilde X$ with a singularity with zero residue at  $p_i$ but with higher order 
poles, generating the dual of $ \bC^N$. Again, their A--periods are normalized to zero, and their B--periods on 
the $k$--th cycle are given by $W_{(i,h),k}$.  
\end{itemize}
Let $(\Pi^S,\, \Pi_{\widetilde X}) \ =\ (\zeta^S_{i,h, \ i =1,\cdots ,s, h = 1,\cdots ,h_i}\, z_{j,\ j= 1,
\cdots ,{\widetilde g}})$
be the corresponding primitives. We 
want a theta function that is a section of ${\pO}_{(\bP^1)^N}(1,\cdots ,1)$ on the compactifications $(\bP^1)^N$
of the fibers of 
$J(X)\,\longrightarrow \,J(\widetilde X)$, and that satisfies the same quasi-periodicity as the ordinary theta
functions on $J(\widetilde X)$. The section will be different from the one given in the case of $\widetilde X$ rational; it is not clear if it decomposes as a product.

Let $\bD_{i,h}\  =\, \sum_{\mu = 1}^{\widetilde g} \frac{W_{(i,h),\mu} }{2\pi\sqrt{-1}}\frac{\partial}{\partial z_\mu}$
be as above, with $W_{(i,h),\mu}$ the $B$-periods of $\zeta^S_{i,h}$ as above. Renumber the forms $\omega^S_{i,h}$ and the primitives $\zeta^S_{i,h}$ as $\omega^S_j,  \zeta^S_j, j= 1,..N$
For $\zeta \,=\, (\zeta_1,\,\cdots ,\,\zeta_N)$,\, $z\,=\, (z_1,\,\cdots ,\, z_{\widetilde g})$,
set, for $I\,=\,\{i_1,\,\cdots ,\, i_\ell\}\,\subset\,\{1,\, \cdots ,\, N\}$:
\begin{itemize} 
\item $I^c\, =\, \{1,\, \cdots ,\, N\}\backslash I$,\, $J_I^C\, =\, I\backslash J$,
\item $\zeta^I \,=\, \zeta_{i_1}\zeta_{i_2}\ldots \zeta_{i_\ell}$,
\item $W_{ I,\alpha}\, =\, W_{ i_1,\alpha}W_{ i_2,\alpha}\ldots W_{ i_\ell,\alpha}$,
\item $\bD_I \,=\, \bD_{i_1}\bD_{i_2}\ldots \bD_{i_\ell}$.
\end{itemize}

\begin{lemma}\label{lem2}
Write the quasi-periodicity relation for $\widetilde{\theta} (z)$ as $\widetilde{\theta}(z+Z_\alpha)\, =\,
\widetilde{\theta}(z) R_\alpha$, with $R_\alpha \,=\, \exp (-2\pi\sqrt{-1} z_\alpha -\pi\sqrt{-1} Z_{\alpha,\alpha})$.
Then  
$$\bD_I(\widetilde{\theta}(z+Z_\alpha))\ =\ [\sum_{J\subset I}(-1)^{|J^c_I|} W_{J_I^c,\alpha} \bD_J\widetilde{\theta}
(z) ]R_\alpha .$$
\end{lemma}

The proof of Lemma \ref{lem2} is a straightforward calculation. 

\begin{lemma}\label{lem3}
Writing $\theta(\zeta,\, z) \,=\, \sum_{I\subset\{1,\cdots ,N\}} F_I(z) \zeta^I$, we have that choosing
$$F_I(z)\ = \  \bD_{I^c}\widetilde{\theta}(z)  $$
gives a $\theta$ satisfying the quasi-periodicity relations on the $A$ and $B$--periods.
\end{lemma}

\begin{proof}The quasi-periodicity relation for $F_I(z)$
gives
$$F_I(z+Z_\alpha)\ =\ (\sum_{J\subset I^c}(-1)^{|J_{I^c}^c|} W_{J_{I^c}^c,\alpha} \bD_J\widetilde{\theta}(z) )
R_\alpha ,$$
and so 
\begin{align}\nonumber
\sum_{I\subset\{1,\cdots ,N\}} F_I(z+Z_\alpha) (\zeta+W_\alpha)^I\ =\ &
\sum_{I\subset\{1,\cdots ,N\}} F_I(z+Z_\alpha)\sum _{K\subset I}  W_{K_I^c,\alpha}\zeta^K \nonumber\\
=\ &\sum_{I\subset\{1,\cdots,N\}} [\sum_{J\subset I^c}(-1)^{|J_{I^c}^c|} W_{J_{I^c}^c,\alpha} \bD_J
\widetilde{\theta}(z) )R_\alpha]\nonumber\\&
\times [ \sum _{K\subset I} W_{K_I^c,\alpha}\zeta^K].\nonumber
\end{align} 
Gathering together all the terms with the same $\bD_J\widetilde{\theta}(z) \zeta^K R_\alpha$,
 $$ \sum_{I| I\supset K, I^c\supset J} (-1)^{|J_{I^c}^c|}W_{J_{I^c}^c,\alpha}W_{K_I^c,\alpha}
\ =\ W_{(J\cup K)^c,\alpha }\sum_{I| I\supset K, I^c\supset J} (-1)^{|J_{I^c}^c|}.$$
If $K\,\neq\, J^c$, the sum is $(1-1)^{N-|K|-|J|} \, =\, 0$; when $K\,=\,J^c\, =\, I$, there is just one term,
which is $1$, giving 
$$\sum_{I\subset\{1,\cdots ,N\}} F_I(z+Z_\alpha) (\zeta+W_\alpha)^I\, =\, 
(\sum_{I\subset\{1,\cdots,N\}}\bD_{I^c}\widetilde{\theta}
(z)\zeta^I)R_\alpha\, =\, (\sum_{I\subset\{1,\cdots,N\}} F_I(z) \zeta^I)R_\alpha$$
which is the standard theta-function periodicity.
\end{proof}

\subsection{Arbitrary singularities} 

Finally, we define our theta function in the case of arbitrary singularities. Here the inverse image in the desingularisation of a
singular point $x_i$ will be a set of points $p_{i,1},\, \cdots ,\,p_{i,j_i}$. As above, we have , then
\begin{itemize}
\item A basis
$\omega_i$,\, $i\,=\,1,\, \cdots ,\,\widetilde g$, of holomorphic differentials on $\widetilde X$; 
\item Forms $\omega^S_{i, j}$,\, $i\,=\,1,\, \cdots ,\,s$,\, $j\,=\, 2,\,\cdots ,\,j_i$, are   with simple poles
at $p_{i, 1}$ and $p_{i, j}$ and residues $+1$ at $p_{i, j}$ and $-1$ at $p_{i,1}$, and

\item Forms $\omega_{i,j,h}^{S}$ having   with higher poles at $p_{i,j}$ and no residues.
\end{itemize}
Both the $\omega^S_{i, j}$ and $\omega_{i,j,h}^{S}$ have zero A--periods. The $\alpha$--th B--periods of
$(\omega^S_{i, j},\,\omega_{i,j,h}^{S},\, \omega_t)$ are $(Y_{(i,j), \alpha},\, W_{(i,j,h),\alpha},\, Z_{t,\alpha})$,
giving as above a matrix of periods
$$\begin{pmatrix} 0&0&\bI\\Y& W& Z\end{pmatrix}. $$
We denote  the $\alpha$--th row of the B--periods by  $Y_\alpha,\, W_\alpha,\, Z_\alpha $; set
$2\pi\sqrt{-1} \nu_j$ to be the $j$--th column of $Y$.
Our Abel map is given by 
$$ \Pi(p)\, =\, (\exp(\xi_{i, j}(p)),\,\zeta_{i,j,h }(p),\,z_k(p))\, =\,( \exp( \int_{p_0}^p  \omega^S_{i, j}),
\, \int_{p_0}^p\omega_{i,j,h}^{S},\, \int_{p_0}^p\omega_k ).$$
For convenience, we index  the $\omega^S_{i, j}$ as  $\omega^S_{\ell}$,\, $\ell\,=\, 1,\cdots ,\,M$,
and the  $\omega_{i,j,h}^{S}$ as  $\omega_\ell^{S,h }$,\, $\ell\, =\, 1,\,\cdots ,\,N$, and define,
for primitives $\xi \,=\, (\xi_1,\, \cdots ,\,\xi_M)$,\, $\zeta \,=\, (\zeta_1,\,\cdots,\, \zeta_N)$,\,
$z\,=\, (z_1,\,\cdots ,\,z_{\widetilde g})$, with the notation of the preceding sections:
$$\theta(\exp(\xi_j),\, \zeta_i,\, z_\mu) \,=\, \sum_{J \subset \{1,\cdots ,M\}, I\subset\{1,\cdots ,N\}}
(\prod_{j\in J} \exp( \xi_j )\zeta^I\bD_{I^c}\widetilde{\theta} (z +\sum_{j\in J}{\nu_j)_\mu}).$$

\begin{lemma}
The quasi-periodicity relations for $\theta(\xi,\, \zeta,\, z)$ are the same as the standard ones for
$\widetilde{\theta} (z)$:
$$ \theta(\exp((\xi +Y_\alpha)_j),\, (\zeta +W_\alpha)_i,\, (z + Z_\alpha)_\mu)\ =\
\theta(\exp(\xi_j),\, \zeta_i,\, z_\mu) \exp (-2\pi\sqrt{-1} z_\alpha -\pi\sqrt{-1} Z_{\alpha,\alpha}).$$
\end{lemma}

\begin{proof}
From the previous section, 
\begin{align*}\sum_{I\subset\{1,..,N\}} &(\zeta +W_\alpha)^I\bD_{I^c}\widetilde{\theta}(z +\sum_{j\in J}\nu_j +Z_\alpha)  \\ =  &[\sum_{I\subset\{1,..,N\}} \zeta^I \bD_{I^c}\widetilde{\theta}(z +\sum_{j\in J}\nu_j  ) ] [\exp
(-2\pi\sqrt{-1} (z+\sum_{j\in J}\nu_j)_\alpha -\pi\sqrt{-1} Z_{\alpha,\alpha}]\end{align*}
and so
\begin{align*}\sum_{J \subset \{1,\cdots ,M\}, I\subset\{1,\cdots ,N\}}&(\prod_{j\in J} (\xi+ Y_\alpha)_j )
(\zeta +W_\alpha)^I\bD_{I^c}\widetilde{\theta}(z +\sum_{j\in J}\nu_j +Z_\alpha)   \\ 
= \ &[\sum_{J \subset \{1,\cdots ,M\}, I\subset\{1,\cdots ,N\}}(\prod_{j\in J} \xi_j ) \zeta^I\bD_{I^c}
\widetilde{\theta} (z+\sum_{j\in J}\nu_j) ][\exp (-2\pi\sqrt{-1} z_\alpha -\pi\sqrt{-1} Z_{\alpha,\alpha}]\end{align*}
as $Y_{\alpha,j}\, =\, 2\pi\sqrt{-1} (\nu_j)_{\alpha}$.
\end{proof}

 \subsection{Translates}
We now have $M+N$ extra dimensions (on top of the $\widetilde g$ dimensions of the Jacobian of the 
desingularisation) in which to shift. For $a \,=\, (a_1,\, \cdots,\,b_M),\ b \,=\, (b_1,\, \cdots,\,b_N)$, define 
define
$$\theta_{a,b,\lambda}(\exp(\xi)_j,\, \zeta_i,\, z_\mu)\ =\ \theta(\exp(\xi_j) a_j^{-1},\, (\zeta_i -b_i ,\, z_\mu-\lambda_\mu). $$

\begin{proposition}
If $q_\rho(a,\, b,\, \lambda), \rho = 1,..,\widetilde{g} +N +M$ are the zeroes of $\theta_{a,b,\lambda}(\Pi(p))$, repeated according to multiplicity,
\begin{align*}\sum_{\rho=1}^{\widetilde{g} +N +M} (\ \xi_j(q_\rho(a,\,b,\, \lambda))),\,\zeta_{i}(q_\rho(a, b, \lambda)),\, z_\mu(q_\rho(a,\,b,\, \lambda)))  \ =\ &    \Bigl( \sum_{i'j'} M_{j',j}(a_{j'}) + \sum_{i' }N_{i',j}(b_{i'}),\\
&  \sum_{j'} P_{j';i}(a_{j'}) + \sum_{i'}Q_{i',i}(b_{i'}), \\
& \kappa+  \lambda  \ \Bigr ) 
\end{align*}  
where $\kappa $ is again  a generalised Riemann constant.
\end{proposition}

\begin{proof} It is the same proof as in our various special cases: we integrate 
$$   d\log (\theta_{a,\,b,\, \lambda})\times(\xi_1,\cdots, \xi_M,\, \zeta_1,\cdots,\, \zeta_N,\,
z_1,\,\cdots ,\, z_{\widetilde g})$$
on the usual contour. 
\begin{itemize} 
\item The $A$ and $B$ cycles contribute the usual $(0,\, \cdots,\, 0,\, -\lambda_1,\,\cdots ,\, -\lambda_{\widetilde g})$, plus constants.
\item The $C$ cycles surround points where the forms $\omega^S_{\ell}, \omega_\ell^{S,h }$ have poles, and   the $\exp(\xi_j)$ have either a simple zero or a simple pole, and the $\zeta_i$ are infinite, with no residues higher order poles. The computation for these cycles reproduces the one in the case of $\widetilde X = \bP^1$. 
\item The $D$-cycles give the values $( \xi_j(q_\rho(a,\,b,\, \lambda)),\,\zeta_{i}(q_\rho(a, b, \lambda)),\, z_\mu(q_\rho(a,\,b,\, \lambda)))$. \end{itemize}

\end{proof}

This gives the flow of divisors on $\widetilde X$ corresponding to a linear flow on the Jacobian of $X$; the 
linear flow on the Jacobian does not give a linear flow on the Lie algebra $q^S(X)$, however; while it is 
conceivable that a more intrinsic compactification or a different definition of theta might give a simpler 
flow, it seems unlikely, as the Abel map combines together contributions from all the 
singular points.



\begin{thebibliography}{AAA}

\bibitem[Ab]{Ab} S. Abenda, On a family of KP multi-line solitons associated to rational degenerations of real
hyperelliptic curves and to the finite non-periodic Toda hierarchy, {\it J. Geom. Phys.} {\bf 119} (2017), 112--138.

\bibitem[AHH]{AHH}  M. R. Adams, J.  Harnad and J. Hurtubise, Isospectral
Hamiltonian Flows in Finite and Infinite Dimensions II. Integration of Flows, 
 {\it Commun. Math.  Phys. } {\bf 134} (1990), 555--585.

\bibitem[AK]{AK} A. Altman and S. Kleiman, Compactifying the Picard scheme, {\it Adv. in Math.}
{\bf 35} (1980), 50--112. 

\bibitem[Be]{B} A. Beauville,  Jacobiennes des courbes spectrales et syst\`emes
Hamiltoniens  compl\`etement int\'egrables, {Acta  Math.} {\bf 164} (1990),  211--235.
 
\bibitem[BR]{BR} I. Biswas and S. Ramanan, An infinitesimal study of the moduli of Hitchin pairs,
{\it Jour. London Math. Soc.} {\bf 49} (1994), 219--231.

\bibitem [Bo]{Bo} F. Bottacin, Symplectic geometry on moduli spaces of stable pairs,
{\it Ann. Sci. \'Ecole Norm. Sup.} {\bf 28} (1995), 391--433.

\bibitem[Ca]{Ca} L. Caporaso, A compactification of the universal Picard variety over the moduli
space of stable curves, {\it J. Amer. Math. Soc.} {\bf 7} (1994), 589--660.

\bibitem[Cl]{Cl} A. Clebsch, \textit{Le\c cons sur la G\'eom\'etrie} (Gauthier-Villars, Paris, 1883), Vol.
3.

\bibitem[Du]{Du} B. A. Dubrovin, Theta Functions and Nonlinear Equations,
{\it Russ. Math. Surv.} {\bf 36} (1981), 11--92.

\bibitem[Ga]{Ga} P. Gaillard, Degenerate Riemann theta functions, Fredholm and wronskian representations of the
solutions to the KdV equation and the degenerate rational case, {\it J. Geom. Phys.} {\bf 161} (2021), Paper No. 104059.

\bibitem[Gi]{Gi} D. Gieseker, A degeneration of the moduli space of stable bundles, {\it J. Differential Geom.}
{\bf 19} (1984), 173--206.

\bibitem[GH]{GH} P. Griffiths and J. Harris, {\it Principles of Algebraic
Geometry},  Wiley, New York (1978).

\bibitem[Ko]{Ko} Y. Kodama, KP solitons and the Riemann theta functions,
{\it Lett. Math. Phys.} {\bf 114} (2024), no. 2, Paper No. 41.

\bibitem[Kr]{K} I. M. Krichever, Methods of Algebraic Geometry in the Theory of Nonlinear
Equations, {\it Russ. Math. Surveys} {\bf 32} (1977), 185--213.

\bibitem[Ma]{Ma} E. Markman, Spectral curves and integrable systems, \textit{Compositio Math.}
\textbf{93} (1994), 255--290.
 
\bibitem[Mu1]{Mu1} D.~B. Mumford, {\it Tata lectures on theta. I}, Progress in Mathematics, 28, Birkh\"auser Boston, Boston, MA, 1983.

\bibitem[Mu2]{Mu2}D.~B. Mumford, {\it Tata lectures on theta. II}, Progress in Mathematics, 43, Birkh\"auser Boston, Boston, MA, 1984.

\bibitem[Mu3]{Mu3}D.~B. Mumford, {\it Tata lectures on theta. III}, reprint of the 1991 original, 
Modern Birkh\"auser Classics, Birkh\"auser Boston, Boston, MA, 2007.

\bibitem[Mu4]{Mu4} D.~B. Mumford, {\it Curves and their Jacobians}, Univ. Michigan Press, Ann Arbor, MI, 1975.
 
\bibitem[OS]{OS} T. Oda and C. Seshadri, Compactifications of the generalized Jacobian variety,
\textit{Trans. Amer. Math. Soc.} {\bf 253} (1979), 1--90.

\bibitem[Ro]{Ro} M. Rosenlicht, Generalized Jacobian varieties, {\it Ann. of Math.} {\bf 59} (1954), 505--530.

\end{thebibliography}
\end{document}